\documentclass{amsart}

\usepackage{amsmath,amssymb,enumitem,graphicx,mathrsfs,tikz,bibentry,tabu}
\usepackage[all]{xy}
\usetikzlibrary{shapes,snakes}
\usepackage[pdfborder={0 0 0},backref=false,colorlinks=true,linktocpage=true]{hyperref}
\PassOptionsToPackage{hyphens}{url}

\newtheorem{theorem}{Theorem}[section]
\newtheorem{proposition}[theorem]{Proposition}
\newtheorem{lemma}[theorem]{Lemma}
\newtheorem{corollary}[theorem]{Corollary}
\theoremstyle{definition}
\newtheorem{definition}[theorem]{Definition}

\newtheorem{question}[theorem]{Question}

\newcommand{\res}{\mathbin{\upharpoonright}}

\newcommand{\seq}[1]{\langle #1 \rangle}

\newcommand{\set}[1]{\{ #1 \}}

\newcommand{\Req}{\mathcal{R}}

\newcommand{\ideal}[1]{\mathcal{#1}}
\newcommand{\collection}[1]{\mathscr{#1}}
\newcommand{\principal}[1]{[#1]^{\leq_T}}

\usepackage{color,soul}
\newcommand{\newauthor}[2]{
\definecolor{#1}{rgb}{#2}
\expandafter\newcommand\csname #1\endcsname[1]%
{{\sethlcolor{#1}\hl{#1 says ``##1.''}}}}

\newauthor{Damir}{0.8,1,0} % Electric Lime
\newauthor{Mariya}{0.992,0.835,0.694} % Light Apricot
\newauthor{Peter}{0.8,0.8,1} % Periwinkle
 
\begin{document}

\title{Genericity for Mathias forcing over general Turing ideals}

\author{Peter A.\ Cholak}
\address{Department of Mathematics\\
University of Notre Dame\\
Notre Dame, Indiana 46556 U.S.A.}
\email{cholak@nd.edu}

\author{Damir D. Dzhafarov}
\address{Department of Mathematics\\
University of Connecticut\\
Storrs, Connecticut U.S.A.}
%196 Auditorium Road\\ Storrs, Connecticut 06269 U.S.A.}
%\curraddr{}
\email{damir@math.uconn.edu}

\author{Mariya I. Soskova}
\address{Faculty of Mathematics and Infromatics\\
Sofia University\\
 5 James Bourchier blvd., Sofia 1164, Bulgaria}
%\curraddr{}
\email{msoskova@fmi.uni-sofia.bg}

\thanks{The second author was partially supported by an NSF Postdoctoral Fellowship and by NSF grant DMS-1400267. The third author was partially supported by an FP7-MC-IOF grant STRIDE (298471). All three authors were additionally partially supported by NSF grant DMS-1101123, which made their initial collaboration possible. The authors thank Rose Weisshaar and the anonymous referee for a number of valuable comments.}

\maketitle

\begin{abstract}
In Mathias forcing, conditions are pairs $(D,S)$ of sets of natural numbers, in which $D$ is finite, $S$ is infinite, and $\max D < \min S$. The Turing degrees and computational characteristics of generics for this forcing in the special (but important) case where the infinite sets $S$ are computable were thoroughly explored by Cholak, Dzhafarov, Hirst, and Slaman~\cite{CDHS-2014}. In this paper, we undertake a similar investigation for the case where the sets $S$ are members of general countable Turing ideals, and give conditions under which generics for Mathias forcing over one ideal compute generics for Mathias forcing over another. It turns out that if $\ideal{I}$ does not contain only the computable sets, then non-trivial information can be encoded into the generics for Mathias forcing over $\ideal{I}$. We give a classification of this information in terms of computability-theoretic properties of the ideal, using coding techniques that also yield new results about introreducibility. In particular, we extend a result of Slaman and Groszek and show that there is an infinite $\Delta^0_3$ set with no introreducible subset of the same degree.
\end{abstract}

\section{Introduction}

Mathias forcing gained prominence in set theory in the article~\cite{Mathias-1977}, for whose author it has come to be named. In a restricted form, it was used even earlier by Soare~\cite{Soare-1969}, to build an infinite set with no subset of strictly higher Turing degree. In computability theory, it has subsequently become a prominent tool for constructing infinite homogeneous sets for computable colorings of pairs of integers, as in Seetapun and Slaman~\cite{SS-1995}, Cholak, Jockusch, and Slaman~\cite{CJS-2001}, and Dzhafarov and Jockusch~\cite{DJ-2009}. It has also found applications in algorithmic randomness, in Binns, Kjos-Hanssen, Lerman, and Solomon~\cite{BKLS-2006}. Dorais~\cite{Dorais-2012} has studied a variant of Mathias forcing that behaves nicely with respect to reverse mathematics.

The conditions in Mathias forcing are pairs $(D,S)$, where $D$ is a finite subset of $\omega$, $S$ is an infinite such subset, and $\max D < \min S$. A condition $(D^*,S^*)$ \emph{extends} $(D,S)$ if $D \subseteq D^* \subseteq D \cup S$ and $S^* \subseteq S$. We think of the finite set $D$ as representing a commitment of information, positive and negative, about a (generic) set to be constructed, and $E$ as representing a commitment of negative information alone.

In computability theory, the interest is typically in Mathias forcing with an imposed effectivity restriction on the conditions used. For instance, in~\cite[Theorem 4.3]{CJS-2001}, the sets $S$ are restricted to be computable, whereas in~\cite[Theorem 2.1]{SS-1995}, the sets are restricted to be members of a Scott set. Many other variants have appeared in the literature. The most general requirement of this form is to restrict the sets $S$ to be members of a fixed countable Turing ideal $\mathcal{I}$.

Our interest in this paper will be in the computability-theoretic properties of generics for the above forcing. A similar analysis for the special case when $\mathcal{I}$ consists just of the computable sets was undertaken by Cholak, Dzhafarov, Hirst, and Slaman~\cite{CDHS-2014}. As we shall see, this situation differs from the general one in a number of important ways. For instance, given a non-computable set $A$, it is always possible to choose a generic for Mathias forcing over the computable ideal that does not compute $A$. We characterize those sets $A$ and ideals $\ideal{I}$ that have this property, and construct ones that do not. Thus, we show there are ideals $\mathcal{I}$ over which every Mathias generic contains some common non-computable information.

The paper is organized as follows. In Section~\ref{S:defns} we give formal definitions of Mathias forcing, and in particular, of how we choose to represent conditions. Section~\ref{S:basics} establishes some basic results and constructions. In Section~\ref{S:comp}, we turn to the computational strength of generics for Mathias forcing, focusing on the problem mentioned above, of which sets are necessarily computed by them. Finally, in Section~\ref{S:ideals}, we compare generics across different ideals.

\section{Background}\label{S:defns}

Throughout, \emph{sets} will refer to subsets of~$\omega$. We shall use standard terminology from computability theory, and refer the reader to Soare~\cite{Soare-TA} for background. For a general introduction to forcing in arithmetic, see Shore~\cite[Section 3]{Shore-TA}.

\begin{definition}
A \emph{(Turing) ideal} is a collection~$\ideal{I}$ of sets closed under~$\leq_T$ and~$\oplus$.
\end{definition}

In this paper we shall be looking at countable ideals only, and so shall avoid explicitly mentioning so henceforth. The simplest ideal is~$COMP$, consisting of all the computable sets, which is of course a sub-ideal of every other ideal. The next simplest example of an ideal is a principal one, consisting of all~$A$-computable sets for some set~$A$. We denote this ideal by~$\principal{A}$ (so that $COMP = \principal{A}$ for any computable set $A$), and note that it can be identified with the set of indices of total~$A$-computable functions, so membership in it is a~$\Sigma^0_3(A)$ relation of sets. Though not every ideal is principal, more general ideals can be presented in a similar way.

\begin{theorem}[Kleene and Post \cite{KP-1954}; Spector \cite{Spector-1956}]
Every ideal~$\ideal{I}$ has an exact pair, i.e., a pair of sets~$A_0,A_1$ so that a set~$S$ belongs to~$\ideal{I}$ if and only if~$S \leq_T A_0$ and~$S \leq_T A_1$.
\end{theorem}

\noindent Membership in such an ideal is then~$\Sigma^0_3(A_0 \oplus A_1)$, and the ideal can be identified with the set of pairs of indices~$\seq{e_0,e_1}$ so that~$\Phi^{A_0}_{e_0}$ and~$\Phi^{A_1}_{e_1}$ are both total and equal. The choice of~$A_0$ and~$A_1$ here is not canonical in any way, but in the sequel we shall implicitly assume that whenever an ideal~$\ideal{I}$ is mentioned a choice of~$A_0$ and~$A_1$ has been made. (Formally, this means different choices of exact pairs give different ideals, even if they are the same as collections of sets.) For simplicity, we suppress mention of~$A_0$ and~$A_1$ when possible: for instance, given~$n \in \omega$, we write~$\Sigma^0_n(\ideal{I})$ instead of~$\Sigma^0_n(A_0 \oplus A_1)$. In a similar spirit, we call~$\ideal{I}$ \emph{arithmetical} if~$A_0$ and~$A_1$ are arithmetical sets.

Recall the following standard terminology.

\begin{definition}
\
\begin{enumerate}
\item A \emph{Mathias condition} is a pair of sets~$(D,S)$ with~$D$ finite,~$S$ infinite, and~$\max D < \min S$.
\item A condition~$(\widetilde{D},\widetilde{S})$ \emph{extends} a condition~$(D,S)$, written~$(\widetilde{D},\widetilde{S}) \leq (D,S)$, if~$D \subseteq \widetilde{D} \subseteq D \cup S$ and~$\widetilde{S} \subseteq S$.
\item A set~$A$ \emph{satisfies} a condition~$(D,S)$ if~$D \subseteq A \subseteq D \cup S$.
\item A set~$A$ \emph{meets} a collection~$\collection{C}$ of conditions if it satisfies a condition in~$\collection{C}$.
\item A set~$A$ \emph{avoids} a collection~$\collection{C}$ of conditions if it satisfies a condition having no extension in~$\collection{C}$.
\end{enumerate}
\end{definition}

We next define the restriction of Mathias forcing to a particular ideal.

\begin{definition}
Let~$\ideal{I}$ be an ideal.
\begin{enumerate}
\item An \emph{$\ideal{I}$-condition} is a Mathias condition~$(D,S)$ such that~$S \in \ideal{I}$.
\item For each~$n \in \omega$, a set~$G$ is \emph{$n$-$\ideal{I}$-generic} if~$G$ meets or avoids every~$\Sigma^0_n(\ideal{I})$-definable collection of~$\ideal{I}$-conditions.
\item A set~$G$ is \emph{$\ideal{I}$-generic} if it is~$n$-$\ideal{I}$-generic for all~$n \in \omega$.
\end{enumerate}
\end{definition}

\noindent Cholak, Dzhafarov, Hirst, and Slaman~\cite[Proposition 2.4]{CDHS-2014} showed that for each~$n \geq 2$, there exists an~$n$-$COMP$-generic set~$G$ with~$G' \leq_T \emptyset^{(n)}$. The argument easily lifts to more general ideals~$\ideal{I}$, yielding~$n$-$\ideal{I}$-generics with~$G' \leq_T \ideal{I}^{(n)}$. We leave the verification to the reader. As discussed in Section 2 of~\cite{CDHS-2014}, it is generally only of interest to consider~$n$-generics for~$n \geq 3$. 

The forcing language over an ideal~$\ideal{I}$ is the usual language of second-order arithmetic augmented by a parameter~$\dot{G}$ for the generic set, but with the modification that it also includes a parameter for~$\ideal{I}$, which is to say, for the exact pair chosen for~$\ideal{I}$. The forcing relation can then be defined in the usual manner (see, e.g.,~\cite{CDHS-2014}, Section 3).

\section{Basic results}\label{S:basics}

We begin by collecting a few basic facts about Mathias generics over arbitrary ideals, which also serve as preliminary results that we shall expound upon in subsequent sections. To begin, we have the following generalization of the well-known fact that Mathias generics are high (see, e.g., \cite{BKLS-2006}, Lemma 6.6, or \cite{CJS-2001}, Section 5.1). For a set~$A$, let~$p_A$ denote the principal function of~$A$.

\begin{proposition}\label{P:domination}
Let~$\ideal{I}$ be an ideal and~$G$ a~$3$-$\ideal{I}$-generic. Then~$p_G$ dominates every function in~$\ideal{I}$.
\end{proposition}

\begin{proof}
Fix any function~$f \in \ideal{I}$, and let~$\collection{C}$ be the collection of all~$\ideal{I}$-conditions~$(D,S)$ such that~$p_S$ majorizes~$f$. Note that~$\collection{C}$ is~$\Sigma^0_3(\ideal{I})$-definable, and that it is in fact dense. (To see this, consider any~$\ideal{I}$-condition~$(D,S)$. Define~$\widetilde{S} \subseteq S$ inductively by letting~$\widetilde{S}(i)$ be the least element of~$S$ larger than~$f(i)$ and all~$\widetilde{S}(j)$ for~$j < i$. Then~$\widetilde{S} \leq_T f \oplus S$, so~$\widetilde{S}$ belongs to~$\ideal{I}$, and~$(D,\widetilde{S})$ belongs to~$\collection{C}$.) We conclude that any~$3$-$\ideal{I}$-generic set~$G$ meets~$\collection{C}$, so~$p_G$ dominates~$f$.
\end{proof}

Our next results concern the computational strength of generics, which shall be the focus of Section~\ref{S:comp}.

\begin{proposition}\label{P:coneavoidance}
Let~$\ideal{I}$ be an ideal and~$A$ any set not in~$\ideal{I}$. Then there exists a Mathias~$\ideal{I}$-generic set that does not compute~$A$.
\end{proposition}

\begin{proof}
We obtain a generic~$G$ by building a sequence of~$\ideal{I}$-conditions~$(D_0,S_0) \geq (D_1,S_1) \geq \cdots$ with~$\lim_s |D_s| = \infty$, and setting~$G = \bigcup_s D_s$. Let~$(D_0,S_0) = (\emptyset, \omega)$, and assume that for some~$s \geq 0$ we have defined~$(D_s,S_s)$.

If~$s$ is even, we work to make~$G$ not compute~$A$. Say~$s = 2e$. Ask if there exists a finite set~$F \subseteq S_s$ and an~$x$ such that~$\Phi^{D_s \cup F}_e(x) \downarrow \neq A(x)$. If so, choose some such~$F$ and~$x$ and let~$D_{s+1} = D_s \cup F$ and~$S_{s+1} = S_s - \{y : y < \varphi^{D_s \cup F}_e(x)\}$. (Here, $\varphi^{D_s \cup F}_e(x)$ denotes the use of the computation $\Phi^{D_s \cup F}_e(x)$.) Now~$G$ will satisfy~$(D_{s+1},E_{s+1})$, and thus~$\Phi^G_e$ will converge on~$x$ and differ from~$A(x)$, ensuring that~$G$ does not compute~$A$ via~$\Phi_e$. If, on the other hand, no such~$F$ and~$x$ exist, let~$(D_{s+1},S_{s+1}) = (D_s,S_s)$. In this case, it cannot be that~$\Phi^G_e$ is total, as otherwise~$S_s$ would compute~$A$, contradicting that~$A \notin \ideal{I}$.

If~$s$ is odd, we work to make~$G$ be~$\ideal{I}$-generic. Say~$s = 2e + 1$, and let~$\collection{C}$ be the~$e$th member in some fixed listing of all arithmetical collections of~$\ideal{I}$-conditions. Let~$(D_{s+1},S_{s+1})$ be any extension of~$(D_s,S_s)$ in~$\collection{C}$, if such exists, and otherwise let~$(D_{s+1},S_{s+1}) = (D_s,S_s)$. Thus,~$G$ will either meet or avoid~$\collection{C}$, as needed.
\end{proof}

We pause here to present a classical computability-theoretic consequence of the above results. Recall that by Martin's high domination theorem, a set~$A$ is high, i.e., satisfies~$A' \geq_T \emptyset''$, if and only if~$p_A$ dominates every computable function. The left-to-right direction of this result relativizes in the following simple form: for any set~$L$, if~$A' \geq_T L''$ then~$p_A$ dominates every~$L$-computable function. The converse is false, as can be shown directly, but we obtain here the following new simple proof of this fact. (The correct relativization of the right-to-left direction of Martin's result is as follows: for any set~$L$, if~$p_A$ dominates every~$L$-computable function then~$(A \oplus L)' \geq_T L''$.)

\begin{corollary}
There exist sets~$L$ and~$G$ such that~$G$ dominates every~$L$-computable function, but~$G' \ngeq_T L''$.
\end{corollary}

\begin{proof}
Suppose not. Let~$L$ be any low c.e.\ set, and~$\ideal{I}$ the principal ideal below~$L$. Let~$M$ be any low c.e.\ set not computable from~$L$, and by Proposition~\ref{P:coneavoidance}, let~$G$ be any~$3$-$\ideal{I}$-generic that does not compute~$M$. By the limit lemma, fix a computable function~$\hat{M}$ in two arguments that approximates~$M$ in the limit.

By Proposition~\ref{P:domination},~$G$ dominates every~$L$-computable function, so by assumption,~$G' \geq_T L''$. But as~$L$ and~$M$ are low,~$L'' \equiv_T \emptyset'' \equiv_T M''$, and so~$G' \geq_T M''$. It follows that~$G$ dominates every~$M$-computable function, and so in particular~$G$ dominates the weak modulus of~$\hat{M}$, defined by
\[
w(x) = (\mu s)[\hat{M}(x,s) = M(x)]
\]
for all~$x$. As~$M$ is c.e., this means~$M$ is computable from~$G$, a contradiction.
\end{proof}

One motivation for us in this paper is the question of which ideals satisfy the converse of Proposition~\ref{P:coneavoidance}: that is, for which ideals~$\ideal{I}$ is it the case that for some~$n$, every~$n$-$\ideal{I}$-generic computes every set~$A \in \ideal{I}$?

\begin{definition}
Fix~$n \in \omega$. An ideal~$\ideal{I}$ is \emph{$n$-generically-coded} if every~$n$-$\ideal{I}$-generic set~$G$ computes every~$A \in \ideal{I}$.
\end{definition}

\noindent We can find many natural examples of generically-coded ideals. Recall the following definition: a function~$f$ is called a \emph{modulus} for a set~$A$ if~$A$ is computable from every function~$g$ that majorizes~$f$. By a result of Solovay~\cite[Theorem 2.3]{Solovay-1978}, a set has a modulus if and only if it is hyperarithmetic.

\begin{lemma}[Folklore]\label{L:Delta2modulus}
If~$A$ is any~$\Delta^0_n$ set,~$n \geq 2$, then~$A$ has a~$(\emptyset^{(n-2)} \oplus A)$-computable modulus.
%Every~$\Delta^0_2$ set has a modulus of the same degree.
\end{lemma}

\begin{proof}
Fix a~$\emptyset^{(n-2)}$-computable function~$\hat{A}$ in two arguments that approximates~$A$ in the limit. Let~$f$ be the~$(\emptyset^{(n-2)} \oplus A)$-computable function mapping each~$n$ to the least~$s \geq n$ such that~$\hat{A}(x,s) = A(x)$ for all~$x \leq n$. Now from any~$g \geq f$ we can compute~$A$ as follows. Given~$n$, we search for an~$s \geq n$ such that~$\hat{A}(n,t) = A(n,s)$ for all~$t$ with~$s \leq t \leq g(s)$. This exists, since for instance any~$s$ after the stage at which~$\hat{A}$ has settled on~$n$ would do. In particular,~$s \leq f(s) \leq g(s)$, so~$\hat{A}(n,s) = \hat{A}(n,f(s))$. But since~$n \leq s$, we have~$\hat{A}(n,f(s)) = A(n)$, so~$\hat{A}(n,s) = A(n)$. We conclude that~$f$ is a modulus for~$A$.
\end{proof}

\begin{proposition}
Let~$\ideal{I}$ be an ideal and~$G$ a~$3$-$\ideal{I}$-generic. If~$\emptyset^{(n-2)} \in \ideal{I}$ for some~$n \geq 2$, then~$G$ computes every~$\Delta^0_n$ set~$A \in \ideal{I}$.
\end{proposition}

\begin{proof}
Assume that~$G$ computes $\emptyset^{(m-2)}$ for some~$m \leq n$, and suppose~$A \in \ideal{I}$ is~$\Delta^0_m$. By Lemma~\ref{L:Delta2modulus},~$A$ has a~$(\emptyset^{(m-2)} \oplus A)$-computable modulus for~$A$. Since $\emptyset^{(m-2)} \leq_T \emptyset^{(n-2)}$ and $\emptyset^{(n-2)} \in \ideal{I}$, it follows that this modulus belongs to $\ideal{I}$. By Proposition~\ref{P:domination},~$G$ dominates this modulus and so computes~$A$. If $m < n$, this means in particular that $G$ computes $\emptyset^{(m-1)}$, so by induction, $G$ computes $\emptyset^{(n-2)}$. Hence, $G$ computes every~$\Delta^0_n$ set in~$\ideal{I}$, as desired.
\end{proof}

\begin{corollary}
If~$\ideal{I}$ is the ideal generated by a union of principal ideals~$\principal{A}$ such that~$A$ satisfies~$\emptyset^{(n)} \leq_T A \leq_T \emptyset^{(n+1)}$ for some~$n \in \omega$, then~$\ideal{I}$ is~$3$-generically-coded.
\end{corollary}

\begin{proof}
Let $A$ be any member of $\ideal{I}$. Then there are sets $A_0,\ldots,A_{k-1}$ such that $A \leq_T A_0 \oplus \cdots \oplus A_{k-1}$, where $\principal{A_0},\ldots,\principal{A}_{k-1}$ are among the principal ideals that generate $\ideal{I}$. Let $n$ be largest such that $\emptyset^{(n)} \leq_T A_i \leq_T \emptyset^{(n+1)}$ for some $i < k$. Then $\emptyset^{(n)} \in \ideal{I}$ and $A \leq_T \emptyset^{(n+1)}$, so by the preceding proposition, we have that every $3$-$\ideal{I}$-generic computes $A$.
\end{proof}

\begin{corollary}\label{C:Delta2coded}
Every~$\Delta^0_2$ ideal is~$3$-generically-coded, as is every principal ideal of the form~$\principal{\emptyset^{(n)}}$,~$n \in \omega$.
\end{corollary}

\noindent The next section is devoted to proving  that the last corollary cannot be extended to general~$\Delta^0_n$ ideals if~$n \geq 3$.

\section{Generics and computation}\label{S:comp}

We address the question of which other ideals are generically-coded. That this is not so for \emph{all} ideals can be seen as follows. Soare~\cite{Soare-1969} exhibited a set~$A$ not computable from any of its co-infinite subsets. Let~$\ideal{I} = \principal{A}$, and let~$G$ be any~$\ideal{I}$-generic satisfying the condition~$(\emptyset,A)$. Then~$G \subseteq A$, and genericity ensures that~$A - G$ is infinite, so~$G$ does not compute~$A$, as desired. Unfortunately, this example leaves a considerable gap between the complexities of ideals that are and are not generically-coded. Indeed, no arithmetical set satisfies Soare's theorem, as the degrees of subsets of infinite arithmetical set are closed upwards (see, e.g., Jockusch~\cite{Jockusch-1973}, Lemma 1.) Hence, the resulting ideal is not arithmetical either.

In what follows, we give an example of a non-generically-coded~$\Delta^0_3$ ideal. We shall employ the following definition.

\begin{definition}
Let~$\ideal{I}$ be an ideal and~$A$ any set. Say a set~$S$ is \emph{$\ideal{I}$-hereditarily uniformly~$A$-computing} if the infinite subsets of~$S$ in~$\ideal{I}$ uniformly compute~$A$, i.e., if there is an~$e \in \omega$ such that~$\Phi^T_e = A$ for every infinite~$T \subseteq S$ in~$\ideal{I}$.
\end{definition}

\noindent The notion is motivated by the concept of uniform introreducibility introduced by Jockusch~\cite[Section 2]{Jockusch-1968a}. Recall that an infinite set is \emph{(uniformly) introreducible} if it is (uniformly) computable from each of its infinite subsets.
%It is not difficult to see that if $A$ has a modulus $f$ then $A$ has an infinite introreducible subset of the same degree as $f$. The converse does not hold, as 
Every degree contains a uniformly  introreducible set, namely, the set of codes of all initial segments of any member of that degree.

%\noindent Recall that a set is \emph{introreducible} if it is computable from each of its infinite subsets, and \emph{uniformly introreducible} if it there is a fixed Turing functional witnessing these computations. It is easy to see that if an infinite set has a modulus of the same degree then it also has an introreducible subset of the same degree.

\begin{proposition}\label{P:comptoinc}
Let~$\ideal{I}$ be an ideal and~$A$ any~$\Delta^0_n(\ideal{I})$ set,~$n \geq 3$. If~$G$ is~$n$-$\ideal{I}$-generic and~$G \geq_T A$ then~$G$ is contained in an~$\ideal{I}$-hereditarily uniformly~$A$-computing set~$S \in \ideal{I}$.
%\begin{enumerate}
%\item~$G \geq_T A$;
%\item~$G$ is contained in an~$\ideal{I}$-hereditarily uniformly~$A$-computing set~$S \in \ideal{I}$.
%\end{enumerate}
\end{proposition}

\begin{proof}
Fix~$e^* \in \omega$ such that~$\Phi^G_{e^*} = A$. Since $n \geq 3$, $G$ must satisfy some condition $(D^{**},S^{**})$ forcing  $\Phi^{\dot{G}}_{e^*}$ to be total, which means that $(D^{**},S^{**})$ avoids the  set  of conditions $(D,S)$ such that
\[
\exists x~\forall~\textnormal{finite sets}~D'~\forall s~[D\subseteq D' \subseteq D\cup S \rightarrow  \Phi_{e^*, s}^{D'}(x)\uparrow].
\]
Let $\collection{C}$ be the collection of all conditions $(D,S)$ such that
\[
\exists x~\exists s~[\Phi^D_{e^*,s}(x) \downarrow \neq A(x)],
\]
which is $\Sigma^0_n(\ideal{I})$-definable since~$A$ is~$\Delta^0_n(\ideal{I})$. By genericity, $G$ must consequently avoid this collection, say via a condition $(D^*,S^*)$. Without loss of generality, we may assume $(D^*,S^*) \leq (D^{**},S^{**})$. Let~$S = D^* \cup S^*$, and let~$e \in \omega$ be the index of the functional that, on input~$x$, searches its oracle for the least finite subset~$F$ such that~$\min F \geq \min S^*$ and
\[
\Phi^{D^* \cup F}_{e^*}(x) \downarrow~ = v
\]
for some value~$v \in \{0,1\}$, and then returns this value. We claim that~$S$ is~$\ideal{I}$-hereditarily uniformly~$A$-computing, as witnessed by~$e$. Since~$G \subseteq S$ and~$S \in \ideal{I}$, this gives the desired conclusion.

To prove the claim, fix any infinite~$T \subseteq S$ in~$\ideal{I}$. If for some~$x$ there were no finite subset~$F$ of~$T$ with~$\min F \geq \min S^*$ and~$\Phi^{D^* \cup F}_{e^*}(x) \downarrow$, then~$(D^*,T \cap S^*)$ would be an~$\ideal{I}$-extension of~$(D^*,S^*)$ (and hence of $(D^{**},S^{**})$) forcing~$\Phi_{e^*}^{\dot{G}}(x) \uparrow$, which cannot be. Hence,~$\Phi^T_e$ is total, by definition of~$e$. Similarly, if for some~$x$ there were such an~$F$ with~$\Phi^{D^* \cup F}_{e^*}(x) \downarrow \neq A(x)$ then~$(D^* \cup F, T^*)$, where~$T^*$ is the set of elements of~$T \cap S^*$ larger than~$\max F$ and the use of~$\Phi^{D^* \cup F}_{e^*}(x)$, would be an extension of $(D^*,S^*)$ in $\collection{C}$. Hence, we also have~$\Phi^T_e = A$.
\end{proof}

The following corollary extends Proposition 2.8 of~\cite{CDHS-2014}, where it appears for the special case of~$\ideal{I} = COMP$. (Compare also with Proposition~\ref{P:coneavoidance}.)

\begin{corollary}\label{C:minpair}
Let~$\ideal{I}$ be an ideal and~$G$ an~$n$-$\ideal{I}$-generic,~$n \geq 3$. Then every~$\Delta^0_n(\ideal{I})$ set computable from~$G$ belongs to~$\ideal{I}$.
\end{corollary}

\begin{proof}
By the preceding proposition, if~$A \leq_T G$ is~$\Delta^0_n$ then in particular~$\ideal{I}$ contains an~$\ideal{I}$-hereditarily uniformly~$A$-computing set, so~$A \in \ideal{I}$.
\end{proof}

We now have the following somewhat technical lemma. For simplicity of notation, we shall assume here that all computations are~$\{0,1\}$-valued, and that all domains of convergence are closed downwards. Given a condition~$(D,S)$, a finite set~$F$, and~$e \in \omega$, we shall write~$F \subseteq \Phi^{(D,S)}_e$ to mean that~$\Phi^D_e$ converges on~$\omega \res \max F + 1$ with use bounded by~$\min S$, and that~$F$ is contained in the finite set given by this computation. So, if~$F \subseteq \Phi^{(D,S)}_e$ then ~$F \subseteq \Phi^{(\widetilde{D},\widetilde{S})}_e$ for every $(\widetilde{D},\widetilde{S}) \leq (D,S)$, and also~$F \subseteq \Phi^U_e$ for any~$U$ satisfying~$(D,S)$ such that~$\Phi^U_e$ is total.

\begin{lemma}\label{L:density}
Let~$\ideal{I}$ be an ideal, $(D^*,S^*)$ an $\ideal{I}$-condition, and $F^*$ a finite set such that~$F^* \subseteq \Phi^{(D^*,S^*)}_e$ for some $e \in \omega$. Fix $i \in \omega$ and let $\collection{C}$ be the collection of conditions~$(D,S)$ satisfying one of the following properties:
\begin{enumerate}
\item $(D,S)$ forces that $\Phi^{\dot{G}}_e$ is not total, not infinite, or not contained in $\dot{G}$;
\item there is an~$x$ such that~$\Phi^{F^* \cup F}_i(x) \uparrow$ for all finite~$F \subseteq S$;
\item there is an~$x \leq \max D$ and a finite~$F \subseteq \Phi^{(D,S)}_e$ with~$\max F^* < \min F$ such that~$\Phi^{F^* \cup F}_i(x) \downarrow \neq D(x)$.
\end{enumerate}
Then $\collection{C}$ is dense below $(D^*,S^*)$. Moreover, if $\ideal{I} = COMP$, then an index for an extension in $\collection{C}$ of a given $(D,S) \leq (D^*,S^*)$, along with which of the above alternatives it satisfies, can be determined uniformly $\emptyset''$-computably from an index for $(D,S)$.
\end{lemma}

\begin{proof}
Call a condition $(D,S)$ \emph{good} if the following facts are true:
\begin{itemize}
\item for all finite $F \subseteq S$ and all $x \in \omega$, if $\Phi^{D \cup F}_e(x) \downarrow~ = 1$ then $x \in D \cup F$;
\item for each $s \in \omega$, there is a finite $F \subseteq S$ and an $x > s$ with $\Phi^{D \cup F}_e(x) \downarrow~ = 1$.
\end{itemize}
Note that if $(D,S)$ is good, then the $F$ and $x$ in the second item above can be found uniformly $S$-computably, and $F$ may be chosen to be non-empty. By iterating this construction, we can thus~$S$-computably construct an infinite subset~$T$ of~$S$ such that~$\Phi_e^{D \cup T}$ is infinite and contained in~$D \cup T$. We denote this $S$-computable subset~$T$ of~$S$ by~$S_D$.

If $(D,S)$ is not good, we can define an extension of it satisfying alternative 1 in the definition of $\collection{C}$, as follows. If there is a finite set $F \subseteq S$ such that $\Phi^{D \cup F}_e(x) \downarrow~ = 1$ for some $x \notin D \cup F$, let $\widetilde{D} = D \cup F$, and let $\widetilde{S}$ consist of the elements of $S$ larger than $x$, the maximum of $D \cup F$, and the use of the computation $\Phi^{D \cup F}_e(x)$. Then $(\widetilde{D},\widetilde{S})$ extends $(D,S)$ and forces that $\Phi^{\dot{G}}_e$ is not contained in $\dot{G}$. If, instead, there is an $s$ such that either $\Phi^{D \cup F}_e(x) \uparrow$ or $\Phi^{D \cup F}_e(x) \downarrow~ = 0$ for all finite sets $F \subseteq S$ and $x > s$, then $(D,S)$ itself forces that $\Phi^{\dot{G}}$ is not total or not infinite. In this case, let $(\widetilde{D},\widetilde{S}) = (D,S)$.

Now to prove the density of $\collection{C}$, fix any condition $(D,S) \leq (D^*,S^*)$. We wish to exhibit an extension of $(D,S)$ in $\collection{C}$, and so, by the preceding observation, we may assume $(D,S)$ and all its extensions are good. (This is equivalent to saying that $(D,S)$ forces that $\Phi^{\dot{G}}_e$ is total, infinite, and contained in $\dot{G}$.)

Since all extensions of $(D,S)$ are good, we can define the following sequence of infinite subsets of~$S$:
\begin{itemize}
\item~$S_0 = S_D$;
\item~$S_1 = \Phi^{D \cup S_0}_e - \{ x^* \}$, where $x^*$ is the least $x > \max D$ in $\Phi^{D \cup S_0}_e$;
\item~$S_2 = (S_1)_D$;
\item~$S_3 = \Phi^{D \cup S_2}_e$.
\end{itemize}
Thus, we have
\[
S \supseteq S_0 \supseteq S_1 \supseteq S_2 \supseteq S_3,
\]
and the number~$x^*$ belongs to~$S_0$ (since $\Phi^{D \cup S_0}_e \subseteq S_0$) but not to~$S_2$. In other words, for~$v \in \{0,1\}$, we have~$S_{2v}(x^*) = 1- v$. There are now two cases.

\medskip
\noindent \emph{Case 1. There is no finite~$F \subseteq S_3$ such that~$\Phi^{F^* \cup F}_i(x^*) \downarrow$.} Then~$(D,S_3)$ is an extension of~$(D,S)$ satisfying alternative 2 in the definition of~$\collection{C}$, witnessed by~$x^*$.

\medskip
\noindent \emph{Case 2. Otherwise.} In this case, choose~$F \subseteq S_3$ and~$v \in \{0,1\}$ such that~$\Phi^{F^* \cup F}_i(x^*) \downarrow~ = v$. Thus,
\[
\Phi^{F^* \cup F}_i(x^*) \downarrow~ = v \neq S_{2v}(x^*) = (D \cup S_{2v})(x^*),
\]
with the last equality holding because $x^*$ is bigger than $\max D$.
%Since $(D,S_{2v}) \leq (D,S) \leq (D^*,S^*)$, we have that $F^* \subseteq \Phi^{(D, S_{2v})}_e$.
Since $(D,S_{2v})$ is good, there is an initial segment $\widetilde{D}$ of~$D \cup S_{2v}$ long enough so that
\begin{itemize}
\item~$x^* \leq \max \widetilde{D}$ and $\widetilde{D}(x^*) = S_{2v}(x^*)$, and
\item~$\Phi^{\widetilde{D}}_e$ converges on~$\omega \res \max F + 1$.
\end{itemize}
Let~$\widetilde{S}$ consist of the elements of~$S_{2v}$ larger than~$\max \widetilde{D}$ and the use of~$\Phi^{\widetilde{D}}_e$ on~$\omega \res \max F + 1$. Thus, $F \subseteq \Phi^{(\widetilde{D},\widetilde{S})}_e$. Finally, since $(D,S)$ extends $(D^*,S^*)$ and is good, we have
\[
F^* \subseteq \Phi^{(D,S)}_e \res \max F^* + 1 \subseteq D,
\]
so~$\max F^* < \min F$ since~$F \subseteq S_3 \subseteq S$. We conclude that $x^*$ and $F$ witness that $(\widetilde{D},\widetilde{S})$ satisfies alternative 3 in the definition of $\collection{C}$.

To complete the proof, suppose $\ideal{I} = COMP$. Now $\emptyset''$ can determine whether or not a condition is good, and can differentiate between Cases 1 and 2. Thus, the only step in the argument that cannot be performed by $\emptyset''$ is the determination of whether or not all extensions of $(D,S)$ are good, which we just assumed above. We can modify the proof to fix this difficulty as follows. Initially, we ask if $(D,S)$ is good, and if not, we can define an appropriate extension of $(D,S)$ that meets alternative 1 in the definition of $\collection{C}$), and are done. Otherwise, we can define $S_0$ and $S_1$ as before. Now we ask if $(D,S_1)$ is good, and if not, we can argue as before, and are again done. Otherwise, we can define $S_2$ and $S_3$, and the rest of the argument is unchanged. It is easy to see that the proof is thus uniform in $\emptyset''$.
\end{proof}

\begin{theorem}\label{T:idealnoncomp}
Let~$\ideal{I}$ be an ideal and~$G$ a~$3$-$\ideal{I}$-generic. Then no~$3$-$\principal{G}$-generic set contained in~$G$ computes~$G$.
\end{theorem}

\begin{proof}
We show that each infinite $G$-computable subset of $G$ has an infinite subset that is computable from $G$, but does not compute $G$ via a prescribed functional. That is, for all $e,i \in \omega$, we build a set $A_{e,i}\leq_T G$ to satisfy the following requirement:
\begin{center}
\begin{tabu}{llX}
$\Req_{e,i}$ & : & if $\Phi^G_e$ is an infinite subset of $G$, then $A_{e,i}$ is an infinite subset of $\Phi^G_e$ and $\Phi_i^{A_{e,i}} \neq G$.
\end{tabu}
\end{center}
In other words, the set~$G$ has no~$\principal{G}$-hereditarily uniformly~$G$-computing subset in the ideal $\principal{G}$. By Proposition~\ref{P:comptoinc}, it follows that no~$3$-$\principal{G}$-generic set satisfying the~$\principal{G}$-condition~$(\emptyset,G)$ can compute~$G$, which is the desired result.

So fix $e$. If $\Phi^G_e$ is not infinite or not contained in $G$, there is nothing to do, and we can let $A_{e,i} = \emptyset$ for all $i$. Otherwise, by~$3$-$\ideal{I}$-genericity of~$G$, we can find an~$\ideal{I}$-condition~$(D^*,S^*)$ satisfied by~$G$ forcing that~$\Phi^{\dot{G}}_e$ is total, infinite, and contained in~$\dot{G}$. In the parlance of Lemma~\ref{L:density}, this means that $(D^*,S^*)$ and all its extensions are good. Let $F^* = \emptyset$, and let $\collection{C}$ be the dense collection of conditions of the lemma. This is a $\Sigma^0_2(\ideal{I})$-definable collection, and so $G$ must meet it, say via some $(D,S) \leq (D^*,S^*)$. By choice of $(D^*,S^*)$, it cannot be that $(D,S)$ satisfies the first alternative in the definition of $\collection{C}$. Hence, one of the following is true:
\begin{itemize}
\item there an~$x$ such that~$\Phi^{F}_i(x) \uparrow$ for all finite~$F \subseteq S$;
\item there is an~$x \leq \max D$ and a finite~$F \subseteq \Phi^{(D,S)}_e$ such that~$\Phi^{F}_i(x) \downarrow \neq D(x)$.
\end{itemize}
In the first case, we may take~$A_{e,i} = \{x \in \Phi^G_e : x \geq \min S\}$, as in this case,~$\Phi^U_i$ cannot be total for any~$U \subseteq S$. In the second case, we may take~$A_{e,i} = F \cup \{x \in \Phi^G_e : x \geq \min S\}$, as then we have
\[
\Phi^{A_{e,i}}_i(x) \downarrow~ = \Phi^F_i(x) \downarrow \neq D(x) = G(x).
\]
Obviously, $A_{e,i} \leq_T \Phi^G_e$, so $A_{e,i}$ is $G$-computable and thus belongs to $\principal{G}$.
\end{proof}

Taking $\ideal{I} = COMP$, and letting $G$ be a $3$-$\ideal{I}$-generic computable in $\emptyset^{(3)}$, the theorem immediately yields the existence of a $\Delta^0_4$ ideal, namely $\principal{G}$, that is not~$n$-generically-coded for any~$n \geq 3$. By Corollary~\ref{C:Delta2coded}, we cannot improve this to $\Delta^0_2$ ideals, but we can improve it to $\Delta^0_3$ ones by a slightly more careful argument. Thus, we get a sharp dividing line in terms of which arithmetical ideals are and are not generically-coded.

\begin{proposition}\label{P:Delta3noncoded}
There is a principal~$\Delta^0_3$ ideal which is not~$n$-generically-coded for any~$n \geq 3$.
\end{proposition}

\begin{proof}
Let~$\ideal{I} = COMP$. The proof is roughly the same as that of Theorem~\ref{T:idealnoncomp}, and we build the sets $A_{e,i}$ to satisfy the same requirements as there. The only difference is that we now need to explicitly construct the set $G$. To this end, we build a sequence of $\ideal{I}$-conditions $(D_0,S_0) \geq (D_1,S_1) \geq \cdots$ by stages, and let $G = \bigcup_s D_s$. At stage $0$, we let $(D_0,S_0) = (\emptyset,\omega)$. We define $A_{e,i}$ at stage $s + 1 = \seq{e,i}$, at which we assume we have already defined $(D_s,S_s)$. Applying Lemma~\ref{L:density} with $(D^*,S^*) = (D_s,S_s)$ and $F^* = \emptyset$, we let $(D_{s+1},S_{s+1})$ be the extension of $(D_s,S_s)$ in $\collection{C}$ given by the lemma. If $(D_{s+1},S_{s+1})$ satisfies the first alternative in the definition of $\collection{C}$, then it forces that $\Phi^{\dot{G}}_e$ is not total, not infinite, or not contained in $\dot{G}$. It is not difficult to check that this fact will be true of any set that satisfies $(D_{s+1},S_{s+1})$, and so in particular of $G$, even though we are not making $G$ even $3$-generic. Thus, $\Phi^G_e$ will not be an infinite subset of $G$ in $\principal{G}$, so we need not worry about it. In this case, we can set $A_{e,i} = \emptyset$. Otherwise, $(D_{s+1},S_{s+1})$ will satisfy one of the other two alternatives in the definition of $\collection{C}$, and then the argument can proceed as in the theorem. By the effectiveness of Lemma~\ref{L:density}, the construction can be performed computably in $\emptyset''$, so $G$ will be $\Delta^0_3$.
\end{proof}

One further consequence of the above is the following purely computability-theoretic result. Slaman and Groszek (unpublished) showed that there is a $\Delta^0_3$ set with no modulus of the same degree. We obtain the analogous result for introreducibility. In fact, it is easy to see that if an infinite set has a modulus of the same degree then it also has an introreducible subset of the same degree, but of course not conversely. Thus, the following is in fact an extension of their result, albeit using a very different argument. The $\Delta^0_3$ bound here is sharp: it is easy to see that every $\Delta^0_2$ set has a modulus of the same degree.

\begin{proposition}
There is a~$\Delta^0_3$ set~$G$ with no infinite introreducible subset of the same degree.
\end{proposition}

\begin{proof}
We show that every infinite $G$-computable subset of $G$ has an infinite subset of its own that does not compute $G$. (The latter subsets will only be computable in the construction, and hence in $\emptyset''$, but not necessarily in $G$.) To this end, we modify the proof of Proposition~\ref{P:Delta3noncoded}. Instead of building, for each $e$, a separate subset $A_{e,i}$ of $\Phi^G_e$ for each $i$ and ensuring that $A_{e,i}$ does not compute $G$ via $\Phi_i$, we build a single subset $A_e$ to satisfy the following requirement:
\begin{center}
\begin{tabu}{llX}
$\Req_e$ & : & if $\Phi^G_e$ is an infinite subset of $G$, then $A_e$ is an infintie subset of $\Phi^G_e$ and $G \nleq_T A_e$.
\end{tabu}
\end{center}
We build the sets $A_e$ by stages along with $G$, and let $F_{e,s}$ be the initial segment to $A_e$ built at stage $s$. Initially, we let $F_{e,0} = \emptyset$ for all $e$. At stage $s + 1 = \seq{e,i}$, we are then given $(D_s,S_s)$ along with $F_{e,s}$, and we assume inductively that $F_{e,s} \subseteq \Phi_e^{(D_s,S_s)}$. We now apply Lemma~\ref{L:density} with $(D^*,S^*) = (D_s,S_s)$ and $F^* = F_{e,s}$ to get an extension $(D,S)$ of $(D_s,S_s)$ in $\collection{C}$. We consider three cases.

\medskip
\noindent \emph{Case 1: $(D,S)$ satisfies the first alternative in the definition of $\collection{C}$.} In this case, we do not have to worry about $\Phi^G_e$ being an infinite subset of $G$, so we also do not have worry about making $A_e$ infinite. We thus set $F_{e,s+1} = F_{e,s}$ and $(D_{s+1},S_{s+1}) = (D,S)$.

\medskip
\noindent \emph{Case 2: $(D,S)$ satisfies the second alternative in the definition of $\collection{C}$ but not the first.} Since $(D,S)$ does not satisfy the first alternative, it does not force that $\Phi^{\dot{G}}_e$ is not infinite, so we can find an $x > \max F_{e,s}$ and an extension $(\widetilde{D},\widetilde{S})$ of $(D,S)$ such that $F_{e,s} \cup \{x\} \subseteq \Phi^{(\widetilde{D},\widetilde{S})}_e$. We let $F_{e,s+1} = F_{e,s} \cup \{x\}$ and $(D_{s+1},S_{s+1}) = (\widetilde{D},\widetilde{S})$. Since $(D,S)$ satisfies the second alternative, we know that $\Phi^U_i$ cannot be total for any set $U$ extending $F_{e,s}$, and so also for any set extending $F_{e,s+1}$.

\medskip
\noindent \emph{Case 3: otherwise.} In this case, $(D,S)$ satisfies the third alternative in the definition of $\collection{C}$, so we can fix a finite set $F \subseteq \Phi^{(D_{s+1},S_{s+1})}_e$ with $\max F_{e,s} < \min F$ witnessing this fact. Since $(D,S)$ does not satisfy the first alternative, we can pass to an extension if necessary to add an element to $F$. Thus, without loss of generality, we may assume $F \neq \emptyset$. We let $F_{e,s+1} = F_{e,s} \cup F$ and $(D_{s+1},S_{s+1}) = (D,S)$.

\medskip
It is easy to verify that for each $e \in \omega$, the construction at stages $s+1 = \seq{e,i}$ ensures the satisfaction of requirement $\Req_e$. And as in our previous result, the entire construction can be carried out by $\emptyset''$, so $G$ is $\Delta^0_3$.
\end{proof}

We finish this section with a converse to Proposition~\ref{P:comptoinc}. We do not know whether the level of genericity below can be improved from~$4$ to~$3$. This is because deciding~$\Sigma^0_3$ facts in general requires Mathias~$4$-genericity (see~\cite{CDHS-2014}, Lemma 3.3), rather than~$3$-genericity as in the case of Cohen forcing.

\begin{proposition}\label{P:inctocomp}
Let~$\ideal{I}$ be an ideal and~$A$ any set. If~$G$ is~$4$-$\ideal{I}$-generic and contained in a set~$S \in \ideal{I}$ such that each infinite subset of~$S$ in~$\ideal{I}$ has an~$\ideal{I}$-hereditarily uniformly~$A$-computing subset in~$\ideal{I}$ then~$G \geq_T A$.
\end{proposition}

\begin{proof}
Suppose every subset of~$S$ in~$\ideal{I}$ has an~$\ideal{I}$-hereditarily uniformly~$A$-computing subset in~$\ideal{I}$. Then in particular,~$A$ belongs to~$\ideal{I}$, so the formula
\begin{equation}\label{E:noncomp}
\exists e~\forall x~[\Phi^{\dot{G}}_e(x) \downarrow~ = A(x)]
\end{equation}
is equivalent to a~$\Sigma^0_3(\ideal{I})$ one. Since~$G$ is~$4$-$\ideal{I}$-generic and contained in~$S$, there must be some~$\ideal{I}$-condition~$(D^*,S^*) \leq (\emptyset,S)$ satisfied by~$G$ that decides this formula. This means~$S^* \subseteq S$, so by passing to an extension if necessary we may assume~$S^*$ is itself~$\ideal{I}$-hereditarily uniformly~$A$-computing subset in~$\ideal{I}$, say witnessed by~$e^* \in \omega$. Let~$e \in \omega$ be the index of the functional defined by
\[
\Phi^U_e = \Phi^{U - D^*}_{e^*}
\]
for all oracles~$U$.
%Thus,~$\Phi^{D^* \cup S^*}_{e} = \Phi^{S^*}_{e^*} = A$.
Now if~$(D^*,S^*)$ forced the negation of \eqref{E:noncomp}, then some extension~$(D^{**},S^{**})$ of it would, for some~$x \in \omega$, either force~$\Phi^{\dot{G}}_e(x) \uparrow$ or else~$\Phi^{\dot{G}}_e(x) \neq A(x)$. In the former case, there could be no finite~$F \subseteq S^{**}$ such that~$\Phi^{D^{**} \cup F}_{e}(x) \downarrow$, so in particular~$\Phi^{D^{**} \cup S^{**}}_{e}$ could not be total. In the latter, it would have to be that~$\Phi^{D^{**}}_e(x) \downarrow \neq A(x)$ with use bounded by~$\min S^{**}$, so also~$\Phi^{D^{**} \cup S^{**}}_{e}(x) \downarrow \neq A(x)$. But~$\Phi^{D^{**} \cup S^{**}}_{e} = \Phi^{S^{**}}_{e^*} = A$, so neither of these cases can be true. We conclude that~$(D^*,S^*)$ forces \eqref{E:noncomp}, and hence by genericity that~$G \geq_T A$.
\end{proof}

\begin{corollary}\label{C:comptoinc}
Let~$\ideal{I}$ be an ideal and~$A$ any~$\Delta^0_n(\ideal{I})$ set,~$n \geq 4$. The following are equivalent for any~$n$-$\ideal{I}$-generic~$G$:
\begin{enumerate}
\item~$G \geq_T A$;
\item~$G$ is contained in an~$\ideal{I}$-hereditarily uniformly~$A$-computing set~$S \in \ideal{I}$. 
\end{enumerate}
\end{corollary}

\begin{proof}
The implication from (1) to (2) is Proposition~\ref{P:comptoinc}, and the implication from (2) to (1) follows by Proposition~\ref{P:inctocomp} and the fact that being~$\ideal{I}$-hereditarily uniformly~$A$-computing is closed under infinite subset.
\end{proof}

\begin{corollary}
Let~$\ideal{I}$ be an ideal and~$A$ any~$\Delta^0_n$ set,~$n \geq 4$. Then the following are equivalent:
\begin{enumerate}
\item every~$n$-$\ideal{I}$-generic~$G$ computes~$A$;
\item every set in~$\ideal{I}$ has an~$\ideal{I}$-hereditarily uniformly~$A$-computing subset in~$\ideal{I}$.
\end{enumerate}
\end{corollary}

\begin{proof}
For the implication from (1) to (2), suppose every~$n$-$\ideal{I}$-generic computes~$A$. Let~$S$ be any member of~$\ideal{I}$, and choose an~$n$-$\ideal{I}$-generic~$G$ satisfying~$(\emptyset,S)$. By Proposition~\ref{P:comptoinc},~$G$ is contained in some~$\ideal{I}$-hereditarily uniformly~$A$-computing set~$S^* \in \ideal{I}$. Then~$S \cap S^*$ is an~$\ideal{I}$-hereditarily uniformly~$A$-computing subset of~$S$ in~$\ideal{I}$. The implication from (2) to (1) is Proposition~\ref{P:inctocomp} and the fact that every generic is contained in some member of~$\ideal{I}$.
\end{proof}

\section{Generics for different ideals}\label{S:ideals}
In this section we examine relationships between Mathias generic sets for different ideals. We ask how the ideals must relate in order for the generics to relate. 
We first establish some basic properties of ideals that will be useful to us in this section. Given two sets~$A$ and~$S$, define
\[
U_{A,S} = \{ p_S(\ulcorner A \res i \urcorner) : i \in \omega \},
\]
where~$\ulcorner A \res i \urcorner$ denotes the canonical index of the finite set~$A \res i$. Thus,~$U_{A,S}$ picks out a subset of~$S$ according to indices of initial segments of~$A$.

\begin{lemma}
Let~$A$ be any set, and suppose~$\ideal{I}$ is an ideal containing a set~$S$. If~$U_{A,S}$ has an infinite subset in~$\ideal{I}$, then~$A$ belongs to~$\ideal{I}$.
\end{lemma}

\begin{proof}
Suppose~$U \subseteq U_{A,S}$ belongs to~$\ideal{I}$. Then~$U$ is a subset of~$S$, so for each~$i$ we can uniformly compute from~$U \oplus S$ an~$x$ such that~$p_U(i) = p_S(x)$. But by the definition of~$U_{A,S}$, this~$x$ must be~$\ulcorner A \res j \urcorner$ for some~$j \geq i$. In other words, given~$i$,~$U \oplus S$ can uniformly find a canonical index for a finite initial segment of~$A$ of size at least~$i$. Since~$A$ is computable from any infinite sequence of its initial segments, we conclude that~$A \leq_T U \oplus S$. Since~$\ideal{I}$ is an ideal and~$U,S \in \ideal{I}$ belong to~$\ideal{I}$, it follows that~$A \in \ideal{I}$.
\end{proof}

\begin{proposition}\label{P:subsets}
If~$\ideal{I} \subseteq \ideal{J}$ are ideals such that every infinite set in~$\ideal{J}$ has an infinite subset in~$\ideal{I}$, then~$\ideal{I} = \ideal{J}$.
\end{proposition}

\begin{proof}
If~$\ideal{J} = \emptyset$ there is nothing to prove, so assume otherwise and fix any infinite~$A \in \ideal{J}$. We shall show that~$A \in \ideal{I}$. By assumption, we may fix an infinite~$S \in \ideal{I}$. Then~$U_{A,S}$ is computable from~$S \oplus A$ and hence belongs to the ideal~$\ideal{J}$. By assumption,~$U_{A,S}$ has an infinite subset in~$\ideal{I}$, so~$A$ belongs to~$\ideal{I}$ by the lemma.
\end{proof}

%\begin{theorem}
%Let~$\ideal{I}$ and~$\ideal{J}$ be ideals. If~$\Sigma^0_n(\ideal{I}) \subseteq \Sigma^0_m(\ideal{J})$ then every~$m$-$\ideal{J}$-generic that computes the fixed exact pair for~$\ideal{I}$ computes an~$n$-$\ideal{I}$-generic.
%\end{theorem}

%The converse of the preceding theorem fails in a strong way.

\begin{proposition}\label{P:smallnocompute}
If~$\ideal{I} \subseteq \ideal{J}$ are arithmetical ideals and~$\ideal{I} \neq \ideal{J}$, then there is an~$\ideal{I}$-generic set that, for some~$n \in \omega$, computes no~$n$-$\ideal{J}$-generic set.
\end{proposition}

\begin{proof}
Let~$\collection{C}$ be the collection of all~$\ideal{J}$-conditions~$(E,T)$ such that~$T$ has no infinite subset in~$\ideal{I}$. We claim this set is dense in the~$\ideal{J}$-conditions. Indeed, fix any~$\ideal{J}$-condition~$(E,T)$, and let~$A$ be any infinite set in~$\ideal{J} - \ideal{I}$. Then~$U_{A,T}$ is an infinite subset of~$T$ in~$\ideal{J}$, and since~$A \notin \ideal{I}$ it follows by the lemma that~$U_{A,T}$ has no infinite subset in~$\ideal{I}$. Hence,~$(E,U_{A,T})$ is an extension of~$(E,T)$ in~$\collection{C}$.

Since~$\ideal{I}$ and~$\ideal{J}$ are arithmetical, there is some~$n$ such that every~$n$-$\ideal{J}$-generic set meets~$\collection{C}$. Our goal, then, is to construct an~$\ideal{I}$-generic set~$G$ such that for all~$e$, either~$\Phi^G_e$ is not an infinite set or~$\Phi^G_e$ avoids~$\collection{C}$. This~$G$ will therefore not compute any~$n$-$\ideal{J}$-generic, as desired. We shall obtain~$G$ in the usual way, as~$\bigcup_s D_s$ from a suitably constructed sequence of~$\ideal{I}$-conditions~$(D_0,S_0) \geq (D_1,S_1) \geq \cdots$. List the elements of~$\collection{C}$ as~$(E_0,T_0),(E_1,T_1),\ldots$. Let~$(D_0,S_0) = (\emptyset,\omega)$, and assume that for some~$s \geq 0$ we have defined~$(D_s,S_s)$.

If~$s$ is even, we work to make~$G$ not compute any~$n$-$\ideal{J}$-generic. More specifically, say~$s = 2\seq{e,i}$. We work to either make~$\Phi^G_e$ not be an infinite set, or to make~$\Phi^G_e$ not meet~$\collection{C}$ via satisfying~$(E_i,T_i)$. To this end, ask if there exists a finite set~$F \subseteq S_s$ and an~$x \notin E_i \cup T_i$ such that~$\Phi^{D_s \cup F}_e(x) \downarrow~ = 1$. If so, choose some such~$F$ and~$x$ and let~$D_{s+1} = D_s \cup F$ and~$S_{s+1} = S_s - \varphi^{D_s \cup F}_e(x)$. Now since~$G$ will satisfy~$(D_{s+1},E_{s+1})$, we have that if~$\Phi^G_e$ is a set, it will contain~$x$ and hence not satisfy~$(E_i,T_i)$. If, on the other hand, no such~$F$ and~$x$ exist, let~$(D_{s+1},S_{s+1}) = (D_s,S_s)$. In this case,~$\Phi^G_e$ cannot be an infinite set, as otherwise~$S_s$ would compute an infinite subset of~$T_i$ and so contradict that~$T_i$ has no infinite subset in~$\ideal{I}$.

If~$s$ is odd, proceed as in the analogous case in the proof of Proposition~\ref{P:coneavoidance}.
\end{proof}

In the~$\Delta^0_2$ setting, we can now extend Proposition~\ref{P:smallnocompute} to non-nested ideals.

\begin{corollary}
If~$\ideal{I}$ and~$\ideal{J}$ are~$\Delta^0_2$ ideals with~$\ideal{J} \nsubseteq \ideal{I}$, then there is an~$\ideal{I}$-generic set that computes no~$3$-$\ideal{J}$-generic.
\end{corollary}

\begin{proof}
Fix~$A \in \ideal{J} - \ideal{I}$. By Corollary~\ref{C:Delta2coded}, every~$3$-$\ideal{J}$-generic set computes~$A$, but by Proposition~\ref{P:coneavoidance}, there is a~$\mathcal{I}$-generic set that does not.
\end{proof}

%\begin{proposition}
%Let~$\ideal{I}$ be an ideal and~$A$ a set in~$\ideal{I}$. If every Mathias~$3$-$\ideal{I}$-generic set computes~$A$ then~$\mathcal{I}$ contains a modulus for~$A$.
%\end{proposition}
%
%\begin{proof}
%By assumption, there must be a~$\ideal{I}$-condition~$(D,S)$ forcing~$\exists e~\Phi_e^G = A$.
%\end{proof} \Mariya{I assume that this is some old statement, as the previous section deals with this kinds of questions already}

The analysis above naturally leads to the following question: 
\begin{question}
If $\ideal{J}\subseteq \ideal{I}$ then does every $\ideal{I}$-generic compute a $3$-$\ideal{J}$-generic? 
\end{question}

This question, although easily stated, is quite difficult to approach. For one thing, one must take into account the complexity of the exact pair representing the smaller ideal. Even if we restrict $\ideal{J}$ to the $\Delta^0_2$ degrees, there are uncountably many ideals of $\Delta^0_2$ degrees, and hence $\ideal{J}$  could turn out to have only very complex exact pairs.   We finish this article with a positive answer to the question above for the simplest case, when $\ideal{J} = COMP$, the ideal of all computable sets.

\begin{theorem}
Let~$\ideal{I}$ be an ideal and  let $n\geq 3$ be a natural number.  Every~$n$-$\ideal{I}$-generic  computes  an~$n$-$COMP$-generic set.
\end{theorem}

We wish to thank Rose Weisshaar, who noticed an omission in the original proof of this theorem. The proof presented below has been corrected.
%More or less in the original we failed to ensure the chosen preconditions were actually conditions.  In the proof below we do this using the sparseness of the generic we are given.

\begin{proof}
We use the following  idea from~\cite[Theorem 5.2]{CDHS-2014}.  For an arbitrary set $A$ we can approximate $A^{(n)}$ by approximating iterations of the jump of $A$ . We let $A^0[s_0] = A\upharpoonright s_0$.  We fix a uniform way to approximate the jump of a set and by induction define for all $m\leq n$: 
\[
A^{(m)}[s_0,\dots, s_{m}] = (A^{(m-1)}[s_0, \dots, s_{m-1}])'[s_{m}].
\]
If $X$ computes $A$ then $X$ computes $A^{(n)}[s_0, \dots, s_{n}]$ for all $n$. Furthermore, for all $x$ and $e$ there are arbitrarily large stages $s_1, \dots, s_{n}$ such that 
\[
A^{(n)} \upharpoonright x = A^{(n)}[s_0,\dots,s_{n}] \upharpoonright x
\]
and for any c.e. set $W$:
\[
W^{A^{(n)}} \upharpoonright x  = W^{A^{(n)}[s_0,\dots,s_{n}]} \upharpoonright x.
\]
These observations justify our adopting the following convention: when we write $A^{(m)}[s_0,\dots, s_{m}]$ we shall always mean that $s_0 < \cdots < s_m$.

%We further inductively define the notion of a correct approximation.  $A\upharpoonright x$ is always \emph{correctly approximated}. If $W$ is an arbitrary c.e. set then $W^{A^{(n)}} [s_0,\dots,s_{n}]\upharpoonright x $ is \emph{correctly approximated} if:
%\begin{enumerate}
%
%\item  $W^{A^{(n)}} \upharpoonright x  = W^{A^{(n)}[s_0,\dots,s_{n}]} \upharpoonright x$;
%\item If $y$ is the largest oracle query in the computation of $W^{A^{(n)}} \upharpoonright x$ then 
%$A^{(n)}[s_0,\dots, s_n]\upharpoonright y$ is also \emph{correctly approximated}. (Here $A^{(n)}$ is viewed as a set that is c.e. in $A^{(n-1)}$).
%\end{enumerate}
%
%
%Note that $A^{(n+1)}$ can decide if $W^{A^{(n)}}$ is correctly approximated. 
 
 Let $G$ be  an $n$-$\ideal{I}$-generic set.  We describe a procedure  which uses $G$ as an oracle to construct an $n$-$COMP$-generic set $E$. The procedure operates as follows: we view  $G$ as a sequence of blocks  each of size $n+1$. Let $\{G_m\}_{m<\omega}$ be a partition of $G$ into finite pieces $G_m$, each of size $n+1$ and with the property that $ \min G_{m+1}< \max G_m$.  We deal with each block $G_m$ in turn and define a finite set $D_m$. The resulting set $E$ will then be $\bigcup_{m<\omega} D_m$. 

We fix a list of requirements: $\{\mathcal{R}_e\}_{e<\omega}$, where each $\mathcal{R}_e$ states the following: 
\begin{quote}
If $R_e = W_e^{\emptyset^{(n)}}$ is a set of conditions then $E$ meets or avoids $R_e$.
\end{quote} 

Consider a block  $G_m = \{s_0, \dots, s_{n}\}$. We use this  piece  first  to produce an approximation  to  the sets $\emptyset^{(p)}$ for all $p\leq n$ using the approximating procedure described above.  Let $\emptyset^{(p)}[m]$  denote the set $\emptyset^{(p)}[s_0, s_1, \dots, s_n]$.
The set $D_m$ is obtained as follows: we construct  a sequence  $\alpha_0, \dots, \alpha_k$, where each $\alpha_i$ has the structure of a computable Mathias condition: it consists of a finite set $D_{\alpha_i}$ and an index  of a partial computable function $e_{\alpha_i}$. This sequence consists of attempts to satisfy the requirements, performed at previous stages, based on previous approximations using earlier blocks. During the current block action we preserve as much of the sequence constructed so far that still seems to consist of actual conditions  and  possibly extend it one step further to satisfy a new requirement.  We will, however, computably ensure that $D_{\alpha_i}\subseteq D_{\alpha_{i+1}}$ for all $i$ and set $D_m = D_{\alpha_k}$. 

The first block $G_0$ will produce  the  sequence with only one element $\alpha_0 = (\emptyset, \mathbb{N})$, or to be precise $(\emptyset, e)$, where $e$ is some fixed index of  the characteristic function of $\mathbb{N}$. Suppose that the  $(m-1)$-st block produces  the sequence $\alpha_0,\dots, \alpha_k$ and  outputs  the set $D_{m-1}$.  We now use the sparseness of $G$ to check whether we still believe that $\alpha_0,\dots, \alpha_k$ is a sequence of conditions: Let $t$ be the largest element in $G_m$ and suppose that $t$ is the $x$-th element of the oracle $G$. Let $S_{\alpha_i, t} = \set{y : \varphi_{e_{\alpha_i}, t}(y)\downarrow = 1}$. For every $i$, such that $0<i<k$ we check whether:
\begin{enumerate}
\item  $\varphi_{e_{\alpha_i}}$ converges on every element $y< x$ in less than $t$ steps  with output $0$ or $1$; 
\item  $S_{\alpha_i, t}$ contains at least $x$ many elements; 
\item  The maximal element in $D_{\alpha_i}$ is smaller than the minimal element in $S_{\alpha_i, t}$; 
\item  $S_{\alpha_i, t}\upharpoonright x\subseteq S_{\alpha_{i-1}, t}$; 
\end{enumerate}

 If we reach an index $j$, such that $\alpha_{j+1}$ does not satisfy the requirements above, then the sequence defined at block $m$ will be $\alpha_0, \dots, \alpha_j, \beta$, where  $\beta = (D_{m-1}, e_{\beta})$ and  $e_{\beta}$ is the index of the  function which outputs $0$ on elements less than  $\max D_{m-1}$ and otherwise behaves like the function with index $e_j$.  

If the sequence defined at block $m-1$ consists of valid Mathias conditions, i.e. for every $i<k$ we have that $\alpha_i$ satisfies the properties above, then  we try  to extend it to meet a   requirement. Using the block $m$ approximation to $\emptyset^{(n)}$, we search for a least number  $e\leq m$  such that
\begin{itemize}
\item  $\mathcal{R}_e$ is not satisfied by any member of the sequence $\alpha_0, \ldots, \alpha_k$;
\item there is an element $(D, e)\in R_e[m]$, such that 
 $(D, e)$ is a computable Mathias condition  which extends $\alpha_k$ according to  $\emptyset''[m]$;
\item $D_{m-1}\subseteq D$ and for every element $x\in D - D_{m-1}$ we have that $\varphi_{e_{\alpha_k}}(x)$ is defined in $t = \max G_m$ many steps and is equal to $1$.
\end{itemize}
If there is such an $e$  then we  set $\alpha_{k+1} = (D, e)$. Otherwise we end the sequence at  $\alpha_k$. This completes the block $m$-action of the procedure.

It remains to show that the constructed set $E = \bigcup_m D_m$ really satisfies each requirement $\mathcal{R}_e$. We use $G$'s genericity for this. We can view the procedure that we just described as a functional $\Gamma$, such that for every finite set $F$,  $\Gamma^{F}$ is a sequence $\alpha_0, \dots, \alpha_k$, where if $m$ is the number of size-$n+1$ blocks that $F$ can be partitioned in then this sequence is obtained by the  block-$m$ action of the  procedure with oracle $F$.

%During the block $m$ actions the functional $\Gamma$  defines the final sequence of Mathias condition using its approximation to $\emptyset''$. We will say that this computation is $2$-correct if the answers are approximated using an initial segment of $\emptyset''$ that is \emph{approximated correctly} as described above. Note that $2$-correct computations produce a sequence of computable Mathias conditions that will  be an initial segment of  every sequence defined at further blocks. 

Assume inductively that there is a condition $(F_0, T_0)$, such that $F_0\subseteq G\subseteq T_0$, such that $\Gamma^{F_0}$ is a monotone sequence of true computable Mathias conditions,  $T_0$ is so sparse that it witnesses this and all requirements $\mathcal{R}_i$ for $i<e$ are satisfied by some member of this sequence.  

For every finite monotone sequence  $\vec \alpha = \alpha_0\dots \alpha_k$,  where $\alpha_i = (D_{\alpha_i}, e_{\alpha_i})$ is a pair of a finite set and an index of a partial computable function, consider the set $C_{\vec \alpha}$ consisting of all  $\mathcal{I}$-conditions $(F,T)$ such that  $T$ is so sparse that  it witnesses that $\vec \alpha$ is a monotone sequence of computable Mathias conditions: for every $i\leq k$ and every $n\geq |F|$, if $t$ is the $n$-th element of $F\cup T$ then:
\begin{enumerate}
\item $\varphi_{e_{\alpha_i}}$ converges on every element $y< x$ in less than $t$ steps  with output $0$ or $1$;  
\item  $S_{\alpha_i, t} = \set{y : \varphi_{e_{\alpha_i}, t} (y)\downarrow = 1}$ contains at least $x$ many elements;  
\item the maximal element in $D_{\alpha_i}$ is smaller than the minimal element in $S_{\alpha_i, t}$; 
\item if  $i>0$ then $S_{\alpha_i, t}\upharpoonright x \subseteq S_{\alpha_{i-1}, t})$. 
\end{enumerate}
Note that if $\vec \alpha$ is a monotone sequence of computable Mathias conditions, then $C_{\vec \alpha}$ is a dense set. In fact, to find an extension of a condition in $C_{\vec \alpha}$ we only need to thin the reservoir, i.e.  for every $(F, T)$ there is an extension $(F, T')\in C_{\vec\alpha}$. 

Consider the  set $\mathcal{C}$ of conditions $(F, T)$ such that $F_0\subseteq F$  and  $\Gamma^{F} = \vec \alpha$ is a  sequence of computable Mathias conditions with a member  in $R_e$ and $T$ is sparse enough to witness this: $(F,T)\in C_{\vec \alpha}$. This set of conditions  is  $\Sigma_n(\ideal{I})$-definable. If $G$ meets $\mathcal{C}$ via $(F,T)$ then  $E$ meets $R_e$ via  one of the conditions in the sequence $\Gamma^F$, because  the sparseness of $T$ guarantees that $\Gamma^F$ is an initial segment of $\Gamma^{G\upharpoonright l}$ for every extension $G\upharpoonright l$ of $F$. 

Suppose that $G$ avoids $\mathcal{C}$ via $(F,T)$. Without loss of generality we may assume $F_0\subseteq F$. Consider   $\Gamma^F$ and let $\alpha_0, \dots, \alpha_k$ be the sequence defined by $\Gamma$ during the computation on the last block of $F$. Let  $j$ be the largest index  such that $\alpha_0, \dots, \alpha_j$ is an initial segment of all further sequences defined at further blocks with oracle $G$. Consider the block of $G$, say $G\upharpoonright l$, where $\alpha_j$ is the last condition  in the computation  $\Gamma^{G\upharpoonright l}$.  Suppose that $\alpha_j$ has an extension $\beta \in R_e$. Let  $(G\upharpoonright l, T^*)$ be an extension of $(F,T)$ that is in $C_{\alpha_0, \dots \alpha_j, \beta}$, such that for the first block  $s_0 < s_1< \cdots < s_n$ of $T^*$ we have that $\beta\in R_e[s_0, s_1, \dots, s_n]$ and furthermore  for every $x\in D_{\beta}$, $x$ can be verified to belong to $D_{\alpha_i}\cup S_{\alpha_i}$ in $s_n$ many steps. 

Now let $(F', T')$ be such that $F' = G\upharpoonright l  \cup{s_1, \dots, s_n}$ and $T'$ is obtained from $ T^*$ by removing all elements less than or equal to $s_n$.  It follows that $(F', T')$ extends $(F,T)$ and belongs to $\mathcal{C}$, contradicting our assumptions.  Thus $\alpha_j$ avoids $R_e$. This completes the proof. 
\end{proof}

%\bibliography{/Users/damir/Documents/Papers/Papers.bib}
%\bibliographystyle{plain}

\end{document}